\documentclass[10pt,reqno]{amsart}
\usepackage{latexsym,amsmath,amssymb,amscd}

\def\today{\ifcase \month \or
   January \or February \or March \or April \or
   May \or June \or July \or August \or
   September \or October \or November \or December \fi
   \space\number\day , \number\year}
   
\makeatletter
  \newcommand\@dotsep{4.5}
  \def\@tocline#1#2#3#4#5#6#7{\relax
     \ifnum #1>\c@tocdepth 
     \else
     \par \addpenalty\@secpenalty\addvspace{#2}%
     \begingroup \hyphenpenalty\@M
     \@ifempty{#4}{%
     \@tempdima\csname r@tocindent\number#1\endcsname\relax
        }{%
         \@tempdima#4\relax
           }%
      \parindent\z@ \leftskip#3\relax \advance\leftskip\@tempdima\relax
      \rightskip\@pnumwidth plus1em \parfillskip-\@pnumwidth
       #5\leavevmode\hskip-\@tempdima #6\relax
       \leaders\hbox{$\m@th
       \mkern \@dotsep mu\hbox{.}\mkern \@dotsep mu$}\hfill
       \hbox to\@pnumwidth{\@tocpagenum{#7}}\par
       \nobreak
        \endgroup
         \fi}
\makeatother 

\begin{document}


\makeatletter
\@addtoreset{figure}{section}
\def\thefigure{\thesection.\@arabic\c@figure}
\def\fps@figure{h,t}
\@addtoreset{table}{bsection}

\def\thetable{\thesection.\@arabic\c@table}
\def\fps@table{h, t}
\@addtoreset{equation}{section}
\def\theequation{
\arabic{equation}}
\makeatother

\newcommand{\bfi}{\bfseries\itshape}

\newtheorem{theorem}{Theorem}
\newtheorem{acknowledgment}[theorem]{Acknowledgment}
\newtheorem{corollary}[theorem]{Corollary}
\newtheorem{definition}[theorem]{Definition}
\newtheorem{example}[theorem]{Example}
\newtheorem{lemma}[theorem]{Lemma}
\newtheorem{notation}[theorem]{Notation}
\newtheorem{proposition}[theorem]{Proposition}
\newtheorem{remark}[theorem]{Remark}
\newtheorem{setting}[theorem]{Setting}

\numberwithin{theorem}{section}
\numberwithin{equation}{section}

\newcommand{\1}{{\bf 1}}
\newcommand{\Ad}{{\rm Ad}}
\newcommand{\Alg}{{\rm Alg}\,}
\newcommand{\Aut}{{\rm Aut}\,}
\newcommand{\ad}{{\rm ad}}
\newcommand{\Borel}{{\rm Borel}}
\newcommand{\Ci}{{\mathcal C}^\infty}
\newcommand{\Cint}{{\mathcal C}^\infty_{\rm int}}
\newcommand{\Cpol}{{\mathcal C}^\infty_{\rm pol}}
\newcommand{\Der}{{\rm Der}\,}
\newcommand{\de}{{\rm d}}
\newcommand{\ee}{{\rm e}}
\newcommand{\End}{{\rm End}\,}
\newcommand{\ev}{{\rm ev}}
\newcommand{\id}{{\rm id}}
\newcommand{\ie}{{\rm i}}
\newcommand{\GL}{{\rm GL}}
\newcommand{\gl}{{{\mathfrak g}{\mathfrak l}}}
\newcommand{\Hom}{{\rm Hom}\,}
\newcommand{\Img}{{\rm Im}\,}
\newcommand{\Ind}{{\rm Ind}}
\newcommand{\Ker}{{\rm Ker}\,}
\newcommand{\Lie}{\text{\bf L}}
\newcommand{\Mt}{{{\mathcal M}_{\rm t}}}
\newcommand{\m}{\text{\bf m}}
\newcommand{\pr}{{\rm pr}}
\newcommand{\Ran}{{\rm Ran}\,}
\renewcommand{\Re}{{\rm Re}\,}
\newcommand{\so}{\text{so}}
\newcommand{\spa}{{\rm span}\,}
\newcommand{\Tr}{{\rm Tr}\,}
\newcommand{\tw}{\ast_{\rm tw}}
\newcommand{\Op}{{\rm Op}}
\newcommand{\U}{{\rm U}}
\newcommand{\UCb}{{{\mathcal U}{\mathcal C}_b}}
\newcommand{\weak}{\text{weak}}

\newcommand{\QuadrHilb}{\textbf{QuadrHilb}}
\newcommand{\LieGr}{\textbf{LieGr}}

\newcommand{\CC}{{\mathbb C}}
\newcommand{\RR}{{\mathbb R}}
\newcommand{\TT}{{\mathbb T}}

\newcommand{\Ac}{{\mathcal A}}
\newcommand{\Bc}{{\mathcal B}}
\newcommand{\Cc}{{\mathcal C}}
\newcommand{\Dc}{{\mathcal D}}
\newcommand{\Ec}{{\mathcal E}}
\newcommand{\Fc}{{\mathcal F}}
\newcommand{\Hc}{{\mathcal H}}
\newcommand{\Jc}{{\mathcal J}}
\newcommand{\Lc}{{\mathcal L}}
\renewcommand{\Mc}{{\mathcal M}}
\newcommand{\Nc}{{\mathcal N}}
\newcommand{\Oc}{{\mathcal O}}
\newcommand{\Pc}{{\mathcal P}}
\newcommand{\Qc}{{\mathcal Q}}
\newcommand{\Sc}{{\mathcal S}}
\newcommand{\Tc}{{\mathcal T}}
\newcommand{\Vc}{{\mathcal V}}
\newcommand{\Uc}{{\mathcal U}}
\newcommand{\Xc}{{\mathcal X}}
\newcommand{\Yc}{{\mathcal Y}}

\newcommand{\Bg}{{\mathfrak B}}
\newcommand{\Fg}{{\mathfrak F}}
\newcommand{\Gg}{{\mathfrak G}}
\newcommand{\Ig}{{\mathfrak I}}
\newcommand{\Jg}{{\mathfrak J}}
\newcommand{\Lg}{{\mathfrak L}}
\newcommand{\Pg}{{\mathfrak P}}
\newcommand{\Sg}{{\mathfrak S}}
\newcommand{\Xg}{{\mathfrak X}}
\newcommand{\Yg}{{\mathfrak Y}}
\newcommand{\Zg}{{\mathfrak Z}}

\newcommand{\ag}{{\mathfrak a}}
\newcommand{\bg}{{\mathfrak b}}
\newcommand{\dg}{{\mathfrak d}}
\renewcommand{\gg}{{\mathfrak g}}
\newcommand{\hg}{{\mathfrak h}}
\newcommand{\kg}{{\mathfrak k}}
\newcommand{\mg}{{\mathfrak m}}
\newcommand{\n}{{\mathfrak n}}
\newcommand{\og}{{\mathfrak o}}
\newcommand{\pg}{{\mathfrak p}}
\newcommand{\sg}{{\mathfrak s}}
\newcommand{\tg}{{\mathfrak t}}
\newcommand{\ug}{{\mathfrak u}}
\newcommand{\zg}{{\mathfrak z}}

\newcommand{\ZZ}{\mathbb Z}
\newcommand{\NN}{\mathbb N}
\newcommand{\BB}{\mathbb B}
\newcommand{\HH}{\mathbb H}

\newcommand{\ep}{\varepsilon}

\newcommand{\hake}[1]{\langle #1 \rangle }

\newcommand{\scalar}[2]{\langle #1 ,#2 \rangle }
\newcommand{\vect}[2]{(#1_1 ,\ldots ,#1_{#2})}
\newcommand{\norm}[1]{\Vert #1 \Vert }
\newcommand{\normrum}[2]{{\norm {#1}}_{#2}}

\newcommand{\upp}[1]{^{(#1)}}
\newcommand{\p}{\partial}

\newcommand{\opn}{\operatorname}
\newcommand{\slim}{\operatornamewithlimits{s-lim\,}}
\newcommand{\sgn}{\operatorname{sgn}}

\newcommand{\seq}[2]{#1_1 ,\dots ,#1_{#2} }
\newcommand{\loc}{_{\opn{loc}}}

\makeatletter
\title[On Weyl calculus in infinitely many variables]{On Weyl calculus in infinitely many variables}
\author{Ingrid Belti\c t\u a 
 and Daniel Belti\c t\u a
}
\address{Institute of Mathematics ``Simion Stoilow'' 
of the Romanian Academy, 
P.O. Box 1-764, Bucharest, Romania}
\email{Ingrid.Beltita@imar.ro}
\email{Daniel.Beltita@imar.ro}
\makeatother

\begin{abstract}
We outline an abstract approach to the pseudo-differential Weyl calculus 
for operators in function spaces in infinitely many variables. 
Our earlier approach to the Weyl calculus for Lie group representations 
is extended to the case of representations associated with 
infinite-dimensional coadjoint orbits. 
We illustrate the approach by the case of infinite-dimensional Heisenberg groups.   
The classical Weyl-H\"ormander calculus is recovered for 
the Schr\"odinger representations of the finite-dimensional Heisenberg groups.
\end{abstract}

\maketitle


\section{Introduction}

The pseudo-differential Weyl calculus which takes into account a magnetic field on $\RR^n$ was recently 
developed in a series of papers including \cite{MP04}, \cite{IMP07}, \cite{MP10}, and \cite{IMP10}. 
We have shown (\cite{BB09a}, \cite{BB09b}, \cite{BB10a}, \cite{BB10d}) that a representation theoretic approach to that calculus 
can lead to a number of improvements such as an extension to the situation of magnetic fields on 
any nilpotent Lie group instead of the abelian group $(\RR^n,+)$ and, more importantly, 
establishing the relationship to the Weyl quantization discussed for instance in \cite{Ca07}. 
The latter point was settled by recovering the magnetic calculus as the Weyl quantization 
for a \emph{finite}-dimensional coadjoint orbit of a Lie group which is in general infinite-dimensional. 

In the present paper we wish to point out that this representation theoretic approach 
can also be applied in the case of certain \emph{infinite}-dimensional coadjoint orbits. 
As a by-product of this method, we provide a generalized version for 
the pseudo-differential calculus developed in \cite{AD96} and \cite{AD98} 
for the differential operators of infinite-dimensional analysis 
(see e.g., \cite{Kuo75}, \cite{Be86}, \cite{DF91}, or \cite{Bg98}). 

\section{An abstract framework for the Weyl calculus}

In this section we develop a version of  
the localized Weyl calculus of \cite{BB09a} and \cite{BB10d}, 
which is general enough for dealing with Weyl quantizations of some infinite-dimensional coadjoint orbits.  

\begin{setting}\label{setting1}
\normalfont
Let $M$ be a locally convex Lie group with Lie algebra $\Lie(M)=\mg$ and smooth exponential map 
$\exp_M\colon\mg\to M$ 
(see \cite{Ne06}), 
and $\pi\colon M\to\Bc(\Yc)$ 
a continuous unitary representation 
on the complex Hilbert space $\Yc$. 
We shall think of the dual space $\mg^*$ as a locally convex space with respect to the $\weak^*$-topology. 
Let $\UCb(\mg^*)$ be the commutative unital $C^*$-algebra of uniformly continuous bounded functions on the locally convex space $\mg^*$ 
and for every $\mu\in\UCb(\mg^*)^*$ define the function 
$$\widehat{\mu}\colon\mg\to\CC,\quad \widehat{\mu}(X)=\langle\mu,\ee^{\ie\langle\cdot,X\rangle}\rangle,$$
where  either of the duality pairings 
$\mg^*\times\mg\to\RR$ and $\UCb(\mg^*)^*\times\UCb(\mg^*)\to\CC$ 
is denoted by~$\langle\cdot,\cdot\rangle$. 
Assume the setting defined by the following data: 
\begin{itemize} 
\item a locally convex real vector space $\Xi$ 
and a Borel measurable map $\theta\colon\Xi\to\mg$, 
\item a locally convex space $\Gamma\hookrightarrow\UCb(\mg^*)^*$ with continuous inclusion map, 
where $\UCb(\mg^*)^*$ is endowed with the $\weak^*$-topology, 
\item a locally convex space $\Yc_{\Xi,\infty}\hookrightarrow\Yc$ 
with continuous inclusion map, 
\end{itemize}
subject to the following conditions: 
\begin{enumerate}
\item\label{quasi-loc_item1} 
The linear mapping  
$$\Fc_\Xi\colon\Gamma\to\UCb(\Xi),\quad 
\mu\mapsto\widehat{\mu}\circ\theta $$
is well defined and injective. 
Let us denote $\Qc_\Xi:=\Fc_\Xi(\Gamma)\hookrightarrow\UCb(\Xi)$ 
and endow it with the topology which makes the \emph{Fourier transform} 
$$\Fc_\Xi\colon\Gamma\to\Qc_\Xi $$
into a linear toplogical isomorphism. 
Note that there also exists the linear toplogical isomorphism 
$(\Fc_\Xi^*)^{-1}\colon\Gamma^*\to\Qc_\Xi^*$.
\item\label{quasi-loc_item2} 
We have the well-defined continuous sesquilinear functional 
$$\Yc_{\Xi,\infty}\times\Yc_{\Xi,\infty}\to\Qc_\Xi,\quad 
(\phi,\psi)\mapsto(\pi(\exp_M(\theta(\cdot)))\phi\mid\psi).$$
\end{enumerate}
\end{setting}

\begin{definition}\label{quasi-loc}
\normalfont
In this framework, 
the \emph{quasi-localized Weyl calculus for $\pi$ along~$\theta$} 
is the linear map $\Op\colon\Gamma^*\to\Lc(\Yc_{\Xi,\infty},\overline{\Yc}_{\Xi,\infty}^*)$ defined by 
\begin{equation}\label{quasi-loc_eq1}
(\Op(a)\phi\mid\psi)
=\langle\underbrace{(\Fc_\Xi^*)^{-1}(a)}_{\hskip20pt\in\Qc_\Xi^*},
\underbrace{(\pi(\exp_M(\theta(\cdot)))\phi\mid\psi)}_{\hskip20pt\in\Qc_\Xi}\rangle
\end{equation}
for $a\in\Gamma^*$ and $\phi,\psi\in\Yc_{\Xi,\infty}$, 
where $\overline{\Yc}_{\Xi,\infty}^*$ denotes the space of antilinear continuous functionals on ${\Yc}_{\Xi,\infty}$. 
\qed
\end{definition}

\begin{remark}\label{quasi-loc_meas}
\normalfont
In the setting of Definition~\ref{quasi-loc}, 
let us assume that the linear functional $(\Fc_\Xi^*)^{-1}(a)\in\Qc_\Xi^*$ 
is defined by a complex Borel measure on $\Xi$ denoted 
in the same way. 
For arbitrary $\phi,\psi\in\Yc_{\Xi,\infty}$,  
the function $(\pi(\exp_M(\theta(\cdot)))\phi\mid\psi)$ 
is uniformly bounded on $\Xi$, 
hence it is integrable with respect to the measure $(\Fc_\Xi^*)^{-1}(a)$ 
and equation \eqref{quasi-loc_eq1} 
takes the form 
\begin{equation}\label{quasi-loc_meas_eq1}
(\Op(a)\phi\mid\psi)
=
\int\limits_\Xi(\pi(\exp_M(\theta(\cdot)))\phi\mid\psi)\de(\Fc_\Xi^*)^{-1}(a)
\end{equation}
which is very similar to the definition of the Weyl-Pedersen calculus 
for irreducible representations of finite-dimensional nilpotent Lie groups 
with the locally convex space $\Xi$ in the role of a predual 
of the coadjoint orbit under consideration
(see for instance \cite{BB09b} and \cite{BB10g}). 

Moreover, for arbitrary $\phi,\psi\in\Yc$ we have $\Vert(\pi(\exp_M(\theta(\cdot)))\phi\mid\psi)\Vert_\infty
\le
\Vert\phi\Vert\cdot\Vert\psi\Vert$, hence by \eqref{quasi-loc_meas_eq1} 
we get 
$\vert(\Op^\theta(a)\phi\mid\psi)\vert\le \Vert(\Fc_\Xi^*)^{-1}(a)\Vert\cdot\Vert\phi\Vert\cdot\Vert\psi\Vert$. 
Thus $\Op(a)\in\Bc(\Yc)$ and $\Vert\Op(a)\Vert\le\Vert(\Fc_\Xi^*)^{-1}(a)\Vert$. 
Here $\Vert(\Fc_\Xi^*)^{-1}(a)\Vert$ denotes the norm of the measure $(\Fc_\Xi^*)^{-1}(a)$ 
viewed as an element of the dual Banach space $\UCb(\Xi)^*$.
\qed
\end{remark}

\begin{remark}\label{quasi-loc_meas_expl}
\normalfont
Assume the setting of Definition~\ref{quasi-loc} again. 
We note that, due to the continuous inclusion map 
$\Gamma\hookrightarrow\UCb(\mg^*)^*$,
every function $f\in\UCb(\mg^*)$ gives rise to 
a functional $a_f\in\Gamma^*$, $a_f(\gamma)=\langle\gamma,f\rangle$ 
for every $\gamma\in\Gamma$. 

Furthermore, let us assume that the function $f\in\UCb(\mg^*)$ 
is the Fourier transform of a Radon measure $\mu\in\Mt(\Xi)$, 
in the sense that 
$f(\cdot)=\int\limits_\Xi\ee^{\ie\langle\cdot,\theta(X)\rangle}\de\mu(X)$. 
Then it is straightforward to check that 
$\Fc_\Xi^*(\mu)=a_f$, 
hence one can use Remark~\ref{quasi-loc_meas} to see that 
$\Op(a_f)=\int\limits_\Xi\pi(\exp_M(\theta(\cdot)))\de\mu $
and $\Op(a_f)\in\Bc(\Yc)$.
\qed
\end{remark}

\subsection*{Preduals for coadjoint orbits}

The following notion recovers the magnetic preduals of \cite{BB09a} 
as very special cases. 

\begin{definition}\label{predual_def}
\normalfont
Let $\gg$ be a nilpotent locally convex Lie algebra 
and pick any coadjoint orbit $\Oc\subseteq\gg^*$ 
of the corresponding Lie group $G=(\gg,\ast)$ defined by the Baker-Campbell-Hausdorff multiplication~$\ast$. 
A \emph{predual} for $\Oc$ is any pair $(\Xi,\theta)$, 
where $\Xi$ is a locally convex real vector space 
and $\theta\colon\Xi\to\gg$ is a continuous linear map 
such that $\theta^*\vert_{\Oc}\colon\Oc\to\Xi^*$ is injective. 
If $\Xi$ is a closed linear subspace of $\gg$ and $\theta$ 
is the inclusion map $\Xi\hookrightarrow\gg$, 
then we say simply that $\Xi$ is a predual for the coadjoint orbit~$\Oc$. 
\qed
\end{definition}

The following statement provides a useful criterion for 
proving that condition~\eqref{quasi-loc_item1} in Definition~\ref{quasi-loc} 
is satisfied in the situation of Fr\'echet preduals 
(see for instance Sect.~7 in Ch.~II of \cite{Sch66}). 
Here $\Mt(\cdot)$ stands for the space of complex Radon measures 
on some topological space. 

\begin{proposition}\label{predual_inj}
Let $\gg$ be a nilpotent locally convex Lie algebra with the Lie group $G=(\gg,\ast)$. 
If $(\Xi,\theta)$ is a predual for the coadjoint orbit $\Oc\subseteq\gg^*$ 
such that $\Xi$ is barreled, 
then the linear mapping  
$\Fc_\Xi\colon\Mt(\Oc)\to\UCb(\Xi)$, 
$\mu\mapsto\widehat{\mu}\circ\theta$
is well defined and injective. 
\end{proposition}

\begin{proof}
See \cite{BB10f}. 
\end{proof}

\subsection*{Flat coadjoint orbits}

We now introduce some terminology that claims its origins in the results of \cite{MW73} 
on representations of \emph{finite-dimensional} nilpotent Lie groups. 

\begin{definition}\label{flat_def}
\normalfont
Let $\gg$ be a nilpotent locally convex Lie algebra with the corresponding Lie group $G=(\gg,\ast)$,  
and denote by $\zg$ the center of $\gg$. 
We shall say that a coadjoint orbit $\Oc$ ($\hookrightarrow\gg^*$) 
is \emph{flat} 
if the coadjoint isotropy algebra at some point $\xi_0\in\Oc$ satisfies the condition $\gg_{\xi_0}=\zg$. 
\qed
\end{definition}

\begin{proposition}\label{flat_prop}
Assume that $\gg$ is a nilpotent locally convex Lie algebra with the Lie group $G=(\gg,\ast)$,  
and denote by $\zg$ the center of $\gg$. 
If $\dim\zg=1$, the coadjoint orbit $\Oc$ is flat, and  $\xi_0\in\Oc$ 
satisfies the condition $\xi_0\vert_{\zg}\not\equiv0$, then $\Ker\xi_0$ is 
a predual for $\Oc$ and 
the mapping 
$\Ker\xi_0\simeq\Oc,\quad X\mapsto(\Ad_G^*X)\xi_0$ 
is a diffeomorphism. 
\end{proposition}

\begin{proof}
See \cite{BB10f}.
\end{proof}

\subsection*{Weyl calculus on flat coadjoint orbits}

\begin{setting}\label{weyl_flat}
\normalfont
Until the end of this section we assume the following setting:
\begin{itemize}
\item $\gg$ is a nilpotent locally convex Lie algebra with the Lie group $G=(\gg,\ast)$; 
\item the topological vector space underlying $\gg$ is barreled 
(for instance, $\gg$ is a Fr\'echet-Lie algebra); 
\item the center $\zg$ of $\gg$ satisfies the condition $\dim\zg=1$; 
\item the coadjoint orbit $\Oc$ ($\hookrightarrow\gg^*$) is flat, 
and  $\xi_0\in\Oc$ 
satisfies $\xi_0\vert_{\zg}\not\equiv0$;  
\item $\pi\colon G\to\Bc(\Yc)$ is an irreducible unitary representation 
such that for every $X\in\zg$ we have 
$\pi(X)=\ee^{\ie\langle\xi_0,X\rangle}\id_{\Yc}$. 
\end{itemize}
In addition we shall denote $\Xi:=\Ker\xi_0$. 
This is a closed hyperplane in $\gg$, hence it is in turn a barreled space 
(see Sect.~7 in Ch.~II of \cite{Sch66}).  

We have the linear isomorphism $\gg/\zg\simeq\Xi$, 
so $\Xi$ is also a nilpotent Lie algebra. 
Let~$\ast_e$ be
the corresponding Baker-Campbell-Hausdorff multiplication. 
We also need the mapping
$s\colon\Xi\times\Xi\to\zg$, $s(X,Y)=X\ast Y-X\ast_e Y$.
\qed
\end{setting}

The following notion is inspired by \cite[Def.~2]{Ma07}. 

\begin{definition}\label{skew_conv}
\normalfont
For arbitrary Radon measures $\mu_1,\mu_2\in\Mt(\Xi)$  
we define their \emph{twisted convolution product} $\mu_1\ast_{\xi_0}\mu_2\in\Mt(\Xi)$ 
as the push-forward of the measure 
$$\ee^{\ie\langle\xi_0,s(X_1,X_2)\rangle}\de(\mu_1\otimes\mu_2)(X_1,X_2)
\in\Mt(\Xi\times\Xi)$$
under the multiplication map 
$\Xi\times\Xi\to\Xi$, $(X_1,X_2)\mapsto X_1\ast_e X_2$. 
\qed
\end{definition}

\begin{theorem}\label{weyl_flat_prop}
Let us assume that we have a locally convex space $\Gamma$ 
such that there exists the continuous inclusion map $\Gamma\hookrightarrow\Mt(\Oc)$. 
Then the following assertions hold: 
\begin{enumerate}
\item\label{weyl_flat_prop_item1} 
The linear mapping  
$\Gamma\to\UCb(\Xi)$,
$\mu\mapsto\widehat{\mu}\vert_{\Xi}$
is well defined and injective, and gives rise to 
the topological linear isomorphism  
$\Fc_\Xi\colon\Gamma\to\Qc_\Xi$ ($\hookrightarrow\UCb(\Xi)$).  
\item\label{weyl_flat_prop_item2} 
If $a_1,a_2\in\Gamma^*$ and
$(\Fc_\Xi^*)^{-1}(a_j)\in\Qc_\Xi^*\cap\Mt(\Xi)$ 
for $j=1,2$, then 
$$\Op(a_1)\Op(a_2)
=\pi((\Fc_\Xi^*)^{-1}(a_1)\ast_{\xi_0}(\Fc_\Xi^*)^{-1}(a_2)).$$ 
\end{enumerate}
\end{theorem}

\begin{proof}
See \cite{BB10f}.
\end{proof}

\section{The special case of infinite-dimensional Heisenberg groups}

\subsection*{Infinite-dimensional Heisenberg groups}

\begin{definition}\label{heisenberg}
\normalfont
Let $\Vc$ be a real Hilbert space endowed with a symmetric injective operator $A\in\Bc(\Vc)$. 

The \emph{Heisenberg algebra} associated with the pair $(\Vc,A)$ is  
$\hg(\Vc,A)=\Vc\dotplus\Vc\dotplus\RR$ with the Lie bracket  
$[(x_1,y_1,t_1),(x_2,y_2,t_2)]=(0,0,(Ax_1\mid y_2)-(Ax_2\mid y_1))$. 
This is a nilpotent Lie algebra.  
The corresponding Lie group $\HH(\Vc,A)=(\hg(\Vc,A),\ast)$ 
is the \emph{Heisenberg group} associated with $(\Vc,A)$, with  
the multiplication given by 
$$(x_1,y_1,t_1)\ast(x_2,y_2,t_2)=
(x_1+x_2,y_1+y_2,t_1+t_2+((Ax_1\mid y_2)-(Ax_2\mid y_1))/2) $$
whenever $(x_1,y_1,t_1),(x_2,y_2,t_2)\in\HH(\Vc,A)$. 
\qed
\end{definition}

\begin{example}
\normalfont
If $\hg(\Vc,A)=\Vc\dotplus\Vc\dotplus\RR$ is a Heisenberg algebra and $\xi_0\in\hg(\Vc,A)^*$ 
as in Definition~\ref{heisenberg}, then it is easy to see that 
$\Xi:=\Vc\times\Vc\times\{0\}$ is a predual for the coadjoint orbit of~$\xi_0$. 
\qed
\end{example}

\begin{remark}
\normalfont
Let $\LieGr$ denote the category of infinite-dimensional Lie groups modeled on locally convex spaces (see \cite{Ne06}) 
and denote by $\QuadrHilb$ the category whose objects are 
the pairs $(\Vc,A)$ 
where $\Vc$ is a real Hilbert space and $A\in\Bc(\Vc)$ is a symmetric, nonnegative, injective operator. 
The morphisms between two objects $(\Vc_1,A_1)$ and $(\Vc_2,A_2)$ 
in $\QuadrHilb$, are defined as 
the continuous linear operators $T\colon \Vc_1\to\Vc_2$ 
satisfying the condition $T^*A_2T=A_1$. 
(Equivalently, $T$ becomes an isometry if $\Vc_j$ 
is endowed with the continuous inner product
$(x,y)\mapsto(Ax\mid y)_{\Vc_j}$ for $j=1,2$.)

Then we have a natural functor 
$\HH\colon\QuadrHilb\to\LieGr$ 
such that the image of any object $(\Vc,A)$ in $\QuadrHilb$ is the corresponding Heisenberg group $\HH(\Vc,A)$ constructed in Definition~\ref{heisenberg}. 
For every morphism $T\colon (\Vc_1,A_1)\to(\Vc_2,A_2)$ 
in $\QuadrHilb$ as above, we have the corresponding morphism 
$$\HH(T)\colon\HH(\Vc_1,A_1)\to  \HH(\Vc_2,A_2), \quad 
(x,y,t)\mapsto(Tx,Ty,t)$$ 
in $\LieGr$. 
\qed
\end{remark}

\subsection*{Gaussian measures and Schr\"odinger representations}
Let $\Vc_{-}$ be a real Hil\-bert space with the scalar product denoted by $(\cdot\mid\cdot)_{-}$. 
For every vector $a\in\Vc_{-}$ and 
every symmetric, nonnegative, injective, \emph{trace-class} operator $K$ on $\Vc_{-}$ 
there exists a unique probability Borel measure~$\gamma$ on $\Vc_{-}$
such that 
$$(\forall x\in\Vc_{-})\quad 
\int\limits_{\Vc_{-}}\ee^{\ie(x\mid y)_{-}}\de\gamma(y)
=\ee^{\ie(a\mid x)_{-}-\frac{1}{2}(Kx\mid x)_{-}} $$
(see for instance \cite[Th.~2.3.1]{Bg98}). 
We also have  
$$a=\int\limits_{\Vc_{-}}y\de\gamma(y)
\quad\text{ and }\quad 
Kx=\int\limits_{\Vc_{-}}(x\mid y)_{-}\cdot(y-a)\de\gamma(y) 
\text{ for all }x\in\Vc_{-}, $$
where the integrals are weakly convergent, and $\gamma$ is called the  
\emph{Gaussian measure with the mean $a$ and the variance $K$}.  

Let us assume that the Gaussian measure $\gamma$ is centered, that is, $a=0$. 
Denote $\Vc_{+}:=\Ran K$ and $\Vc_0:=\Ran K^{1/2}$ endowed 
with the scalar products $(Kx\mid Ky)_{+}:=(x\mid y)_{-}$ and 
$(K^{1/2}x\mid K^{1/2}y)_0:=(x\mid y)_{-}$, respectively, for all $x,y\in\Vc_{-}$, 
which turn the linear bijections 
$K\colon\Vc_{-}\to\Vc_{+}$ and $K^{1/2}\colon\Vc_{-}\to\Vc_0$ 
into isometries. 
We thus get the real Hilbert spaces
$$\Vc_{+}\hookrightarrow\Vc_0\hookrightarrow\Vc_{-} $$
where the inclusion maps are Hilbert-Schmidt operators, 
since $K^{1/2}\in\Bc(\Vc_{-})$ is so. 
Also, the scalar product of $\Vc_0$ extends to a duality pairing $(\cdot\mid\cdot)_0\colon\Vc_{-}\times\Vc_{+}\to\RR$. 

We also recall that for every $x\in\Vc_{+}$ the translated measure $\de\gamma(-x+\cdot)$ 
is absolutely continuous with respect to $\de\gamma(\cdot)$ 
and we have the Cameron-Martin formula
$$\de\gamma(-x+\cdot)=\rho_x(\cdot)\de\gamma(\cdot) 
\quad\text{ with }\rho_x(\cdot)=\ee^{(\cdot\mid x)_0-\frac{1}{2}(x\mid x)_0}.$$  (This actually holds true for every $x\in\Vc_0$; 
see for instance \cite[Cor.~2.4.3]{Bg98}.)

\begin{definition}\label{sch_def}
\normalfont
Let $(\Vc_{+},A)$ be an object in the category $\QuadrHilb$ such that $A\colon\Vc_{+}\to\Vc_{+}$ is a nonnegative, symmetric, injective, trace-class operator. 
Denote the scalar product of $\Vc_{+}$ by $(x,y)\mapsto (x\mid y)_{+}$ 
and let $\Vc_0$ and $\Vc_{-}$ be the completions of $\Vc_{+}$ with respect to the scalar products  
$(x,y)\mapsto (x\mid y)_0:=(A^{1/2}\mid A^{1/2}y)$ 
and $(x,y)\mapsto (x\mid y)_{-}:=(Ax\mid Ay)$, respectively.  
Then the operator $A$ has a unique extension to a nonnegative, symmetric, injective, trace-class operator $K\in\Bc(\Vc_{-})$ 
such that the above setting is recovered 
(see for instance \cite[Ch.~1, \S 1]{Be86}), hence we get the centered Gaussian measure $\gamma$ on $\Vc_{-}$ with the variance $K$. 

On the other hand, we can construct the Heisenberg group $\HH(\Vc_{+},A)$. 
The \emph{Schr\"odinger representation} 
$\pi\colon\HH(\Vc_{+},A)\to\Bc(L^2(\Vc_{-},\gamma))$ is defined by 
$$\pi(x,y,t)\phi=\rho_x(\cdot)^{1/2}
\ee^{\ie(t+(\cdot\mid y)_0+\frac{1}{2}(x\mid y)_0)}\phi(-x+\cdot) $$
whenever $(x,y,t)\in\HH(\Vc_{+},A)$ and $\phi\in L^2(\Vc_{-},\gamma)$. 
This is a continuous unitary irreducible representation 
of the Heisenberg group $\HH(\Vc_{+},A)$; 
see \cite[Th.~5.2.9 and 5.2.10]{Hob06}.
\qed
\end{definition}

\begin{remark}\label{sch_neeb}
\normalfont
We note that more general Schr\"odinger representations of infinite-dimensional Hei\-senberg groups are described in \cite[Prop.~II.4.6]{Ne00} by using cocycles and reproducing kernel Hilbert spaces. 
\qed
\end{remark}

\begin{remark}\label{sch_irred}
\normalfont
One way to see that the representation 
$\pi\colon\HH(\Vc_{+},A)\to\Bc(L^2(\Vc_{-},\gamma))$ 
of Definition~\ref{sch_def} is irreducible 
is the following 
For every integer $n\ge1$ let $\Vc_{n,+}$ denote the spectral space for $A$ corresponding to the interval $[1/n,\infty)$. 
That is, $\Vc_{n,+}$ is spanned by the eigenvectors of $A$ corresponding to eigenvalues $\ge 1/n$. 
Since $A$ is a compact operator, it follows that $\dim\Vc_{n,+}<\infty$. 
We have 
$$\Vc_{1,+}\subseteq\Vc_{2,+}\subseteq\cdots\subseteq
\bigcup\limits_{n\ge1}\Vc_{n,+}\subseteq\Vc_{+}$$
and $\bigcup\limits_{n\ge1}\Vc_{n,+}$ is a dense subspace of $\Vc_{+}$. 
Let us denote by $A_n$ the restriction of $A$ to $\Vc_{n,+}$. 
Then $\HH(\Vc_{n,+},A_n)$ is a finite-dimensional Heisenberg group, 
hence it is well known that its Schr\"odinger representation 
$\pi_n\colon\HH(\Vc_{n,+},A_n)\to\Bc(L^2(\Vc_{n,-},\gamma_n))$ 
is irreducible, where $\gamma_n$ is the Gaussian measure on the finite-dimensional space $\Vc_{n,-}$ obtained out of 
the pair $(\Vc_{+,n},A_n)$ 
by the construction outlined at the very beginning of Definition~\ref{sch_def}. 
Note that 
$$\HH(\Vc_{1,+},A_1)\subseteq\HH(\Vc_{2,+},A_2)\subseteq\cdots\subseteq
\bigcup\limits_{n\ge1}\HH(\Vc_{n,+},A_n)=:\HH(\Vc_{\infty,+},A_\infty)
\subseteq \HH(\Vc_{+},A) $$
and $\HH(\Vc_{\infty,+},A_\infty)$ is a dense subgroup of $\HH(\Vc_{+},A)$, 
hence the Schr\"odinger representation $\pi\colon\HH(\Vc_{+},A)\to\Bc(L^2(\Vc_{-},\gamma))$ is irreducible if and only if so is its restriction $\pi\vert_{\HH(\Vc_{\infty,+},A_\infty)}$. 

On the other hand, if we denote by $\1_n$ the function identically equal to 1 on the orthogonal complement $\Vc_{n+1,+}\ominus\Vc_{n,+}$, then it is straightforward to check that the operator 
$$L^2(\Vc_{n,-},\gamma_n)\to L^2(\Vc_{n+1,-},\gamma_{n+1}), 
\quad f\mapsto f\otimes \1_n$$
is unitary and intertwines the representations $\pi_n$ and $\pi_{n+1}$. 
We can thus make the sequence of representations $\{\pi_n\}_{n\ge1}$ into an inductive system of irreducible unitary representations and then their inductive limit $\pi\vert_{\HH(\Vc_{\infty,+},A_\infty)}
=\mathop{\rm ind}\limits_{n\to\infty}\pi_n$ is irreducible 
(see for instance \cite{KS77}). 
As noted above, this implies that the Schr\"odinger representation 
$\pi\colon\HH(\Vc_{+},A)\to\Bc(L^2(\Vc_{-},\gamma))$ 
of Definition~\ref{sch_def} is irreducible. 
\qed
\end{remark}

The infinite-dimensional pseudo-differential calculus of \cite{AD96} and \cite{AD98} can be recovered as a quasi-localized Weyl calculus for 
the Schr\"odinger representations introduced in Definition~\ref{sch_def} above. 
Compare for instance \cite[Prop.~3.7]{AD98}.

\subsection*{Acknowledgment} 
The second-named author acknowledges partial financial support
from the Project MTM2007-61446, DGI-FEDER, of the MCYT, Spain, 
and from the CNCSIS grant PNII - Programme ``Idei'' (code 1194).


\begin{thebibliography}{100000}

\bibitem[AD96]{AD96}
\textsc{S.~Albeverio, A.~Daletskii}, 
Asymptotic quantization for solution manifolds of some infinite-dimensional Hamiltonian systems. 
\textit{J. Geom. Phys.} \textbf{19} (1996), no.~1, 31--46.

\bibitem[AD98]{AD98}
\textsc{S.~Albeverio, A.~Daletskii}, 
Algebras of pseudodifferential operators in $L_2$ given by smooth measures on Hilbert spaces. 
\textit{Math. Nachr.} \textbf{192} (1998), 5--22.

\bibitem[BB09a]{BB09a} 
\textsc{I.~Belti\c t\u a, D.~Belti\c t\u a}, 
Magnetic pseudo-differential Weyl calculus on nilpotent Lie groups. 
{\it Ann. Global Anal. Geom.} \textbf{36} (2009), no. 3, 293--322. 

\bibitem[BB09b]{BB09b}
\textsc{I.~Belti\c t\u a, D.~Belti\c t\u a}, 
A survey on Weyl calculus for representations of nilpotent Lie groups. 
In: P.~Kielanowski, S.T.~Ali, A.~Odzijewicz, M.~Schlichenmeier, Th.~Voronov (eds.), 
{\it XXVIII Workshop on Geometrical Methods in Physics}, 
AIP Conf. Proc., Amer. Inst. Phys., 1191, Melville, NY, 2009, 
pp.~7--20.


\bibitem[BB10a]{BB10a}
\textsc{I.~Belti\c t\u a, D.~Belti\c t\u a}, 
Uncertainty principles for magnetic structures on certain coadjoint orbits. 
{\it J. Geom. Phys.} \textbf{60} (2010), no.~1, 81--95.

\bibitem[BB10b]{BB10b}
\textsc{I.~Belti\c t\u a, D.~Belti\c t\u a}, 
Smooth vectors and Weyl-Pedersen calculus 
for representations of nilpotent Lie groups. 
\textit{An. Univ. Bucure\c sti Mat.} \textbf{58} (2010), no.~1, 17--46. 
(Preprint arXiv:0910.4746v1 [math.RT].).

\bibitem[BB10c]{BB10c} 
\textsc{I.~Belti\c t\u a, D.~Belti\c t\u a}, 
Modulation spaces of symbols for representations of nilpotent Lie groups. 
\textit{J. Fourier Anal. Appl.} (to appear). 

\bibitem[BB10d]{BB10d}
\textsc{I.~Belti\c t\u a, D.~Belti\c t\u a}, 
Continuity of magnetic Weyl calculus. 
\textit{Preprint} arXiv: 1006.0585v3 [math.AP].

\bibitem[BB10e]{BB10e}
\textsc{I.~Belti\c t\u a, D.~Belti\c t\u a}, 
Algebras of symbols associated with the Weyl calculus for Lie group representations. 
\textit{Preprint} arXiv: 1008.2935v1 [math.FA].

\bibitem[BB10f]{BB10f}
\textsc{I.~Belti\c t\u a, D.~Belti\c t\u a}, 
Weyl quantization for infinite-dimensional coadjoint orbits. 
\textit{Preprint}, 2010. 

\bibitem[BB10g]{BB10g}
\textsc{I.~Belti\c t\u a, D.~Belti\c t\u a},
\textit{Weyl Calculus for Lie Group Representations} 
(forthcoming monograph).
 
\bibitem[Be86]{Be86}
\textsc{Yu.M.~Berezanski\u\i}, 
\textit{Selfadjoint Operators in Spaces of Functions of Infinitely Many Variables}. 
Translations of Mathematical Monographs, 63. American Mathematical Society, Providence, RI, 1986.

\bibitem[Bg98]{Bg98}
\textsc{V.I.~Bogachev}, 
\textit{Gaussian Measures}. 
Mathematical Surveys and Monographs, 62. 
American Mathematical Society, Providence, RI, 1998.

\bibitem[Ca07]{Ca07} 
\textsc{B.~Cahen}, 
Weyl quantization for semidirect products.  
\textit{Differential Geom. Appl.} \textbf{25} (2007), no.~2, 177--190.

\bibitem[DF91]{DF91}
\textsc{Yu.L.~Dalecky, S.V.~Fomin}, 
\textit{Measures and Differential Equations in Infinite-Dim\-ensional Space}. Mathematics and its Applications (Soviet Series), 76. 
Kluwer Academic Publishers Group, Dordrecht, 1991.

\bibitem[H\"ob06]{Hob06}
\textsc{M.G.~H\"ober}, 
\textit{Pseudodifferential Operators on Hilbert Space Riggings with Associated $\Psi^*$-Algebras and Generalized H\"ormander Classes}. 
PhD Thesis, Johannes Gutenberg-Universit\"at im Mainz, 2006. 


\bibitem[IMP07]{IMP07}
\textsc{V.~Iftimie, M.~M\u antoiu, R.~Purice}, 
Magnetic pseudodifferential operators. 
{\it Publ. Res. Inst. Math. Sci.} \textbf{43} (2007), no.~3, 585--623.

\bibitem[IMP10]{IMP10}
\textsc{V.~Iftimie, M.~M\u antoiu, R.~Purice}, 
Commutator criteria for magnetic pseudodifferential operators
{\it Comm. Partial Differential Equations} \textbf{35} (2010), no.~6, 1058–-1094.

\bibitem[KS77]{KS77}
\textsc{V.I.~Kolomycev, Ju.S.~Samoilenko}, 
Irreducible representations of inductive limits of groups. (Russian) 
\textit{Ukrain. Mat. Zh.} \textbf{29} (1977), no. 4, 526--531. 

\bibitem[Kuo75]{Kuo75}
\textsc{H.H.~Kuo}, 
\textit{Gaussian Measures in Banach Spaces}. 
Lecture Notes in Mathematics, Vol.~463. Springer-Verlag, Berlin-New York, 1975.

\bibitem[Ma07]{Ma07}
\textsc{J.-M.~Maillard}, 
Explicit star products on orbits of nilpotent Lie groups 
with square integrable representations. 
\textit{J. Math. Phys.} \textbf{48} (2007), no.~7, 073504.

\bibitem[MP04]{MP04}
\textsc{M.~M\u antoiu, R.~Purice}, 
The magnetic Weyl calculus. 
{\it J. Math. Phys.} \textbf{45} (2004), no.~4, 1394--1417.

\bibitem[MP10]{MP10}
\textsc{M.~M\u antoiu, R.~Purice}, 
The modulation mapping for magnetic symbols and operators. 
\textit{Proc. Amer. Math. Soc.} \textbf{138} (2010), no.~8, 2839--2852.

\bibitem[MW73]{MW73}
\textsc{C.C.~Moore, J.A.~Wolf}, 
Square integrable representations of nilpotent groups.  
\textit{Trans. Amer. Math. Soc.} \textbf{185} (1973), 445--462. 

\bibitem[Ne00]{Ne00} 
\textsc{K.-H.~Neeb},
\textit{Holomorphy and Convexity in Lie Theory}. 
de Gruyter Expositions in Mathematics, 28. 
Walter de Gruyter \& Co., Berlin, 2000.

\bibitem[Ne06]{Ne06} 
\textsc{K.-H.~Neeb}, 
Towards a Lie theory of locally convex groups. 
\textit{Japanese J. Math.} \textbf{1} (2006), no.~2, 291--468.

\bibitem[Sch66]{Sch66}
H.H.~Schaefer, 
\textit{Topological Vector Spaces}, 
Macmillan, New York, 1966. 

\end{thebibliography}
\end{document}